\documentclass[12pt,fleqn]{amsart}

\usepackage[T1]{fontenc}
\usepackage[utf8]{inputenc}
\usepackage[l2tabu,orthodox]{nag}
\usepackage[sc,osf]{mathpazo}   
\usepackage[euler-digits,euler-hat-accent,small]{eulervm}

\let\originalleft\left
\let\originalright\right
\renewcommand{\left}{\mathopen{}\mathclose\bgroup\originalleft}
\renewcommand{\right}{\aftergroup\egroup\originalright}

\usepackage{enumitem}
\setlist[enumerate,1]{label=(\roman*)}
\setlist[enumerate,2]{label=\alph*.}

\usepackage[margin=1.1in]{geometry}

\usepackage{amssymb}

\let\emptyset\varnothing

\newcommand{\C}{\mathbb{C}}
\newcommand{\B}{\mathbb{B}}

\newcommand{\bH}{\mathbb{H}}
\newcommand{\delH}{\partial\mathbb{H}}

\renewcommand{\Im}{\operatorname{Im}}

\newcommand{\cL}{\mathcal{L}}
\newcommand{\transpose}[1]{{#1}^\intercal} 

\usepackage{mathtools}
\DeclarePairedDelimiter\abs{\lvert}{\rvert}%
\DeclarePairedDelimiter\norm{\lVert}{\rVert}%
\DeclarePairedDelimiter\braces{\lbrace}{\rbrace}%
\newcommand{\set}[1]{\braces{\, #1 \, }}

\newcommand{\suchthat}{\mid}

\newcommand{\interior}[1]{%
	{\kern0pt#1}^{\mathrm{o}}%
} 

\DeclareMathOperator{\weightedorder}{o_{wt}}
\DeclareMathOperator{\adj}{adj}
\DeclareMathOperator{\rank}{rank}
\DeclareMathOperator{\Rk}{Rk}

\DeclareMathOperator{\Aut}{Aut}
\usepackage{xspace}

\DeclarePairedDelimiter\inner{\langle}{\rangle}%

\makeatletter
\let\oldabs\abs
\def\abs{\@ifstar{\oldabs}{\oldabs*}}
\let\oldnorm\norm
\def\norm{\@ifstar{\oldnorm}{\oldnorm*}}
\makeatother
\makeatletter
\let\oldinner\inner
\def\inner{\@ifstar{\oldinner}{\oldinner*}}

\newcommand\restrict[2]{{
	\left.\kern-\nulldelimiterspace 
	#1 
	\vphantom{\big|} 
	\right|_{#2} 
}}

\usepackage[thinc]{esdiff}

\usepackage[
  colorlinks,
  citecolor=blue,
  linkcolor=blue,
  urlcolor=blue,
]{hyperref}

\usepackage[msc-links]{amsrefs}

\usepackage{amsthm}
\usepackage[noabbrev,capitalize,nameinlink]{cleveref} 

\theoremstyle{plain}
\newtheorem{theorem}{Theorem}[section]
\newtheorem{corollary}[theorem]{Corollary}
\newtheorem{lemma}[theorem]{Lemma}
\newtheorem{proposition}[theorem]{Proposition}

\theoremstyle{definition}
\newtheorem{definition}[theorem]{Definition}

\theoremstyle{remark}

\newtheorem{notation}{Notation}

\author{Abdullah Al Helal}
\title{Degree of Ball Maps with Maximum Geometric Rank}
\date{}
\address{Department of Mathematics, Oklahoma State University, Stillwater, OK 74078-5061}
\email{ahelal@okstate.edu}

\subjclass[2020]{32H35, 32A08, 32H02}
\keywords{rational maps, proper holomorphic mappings}

\begin{document}

\begin{abstract}
    This work focuses on the degree bound of maps between balls 
    with maximum geometric rank and minimum target dimension where this geometric rank occurs.
    Specifically, we show that rational proper maps between $\B_n$ and $\B_N$ with $n \geq 2$, $N = \frac{n(n+1)}{2}$, and geometric rank $n-1$ cannot have a degree of more than $n+1$.
\end{abstract}

\maketitle

\section{Introduction}
\label{sec:intro}

Rational proper maps between balls has intrigued mathematicians for a long time 
since Fatou~\cite{fatou-1923-fonctions} proved that 
proper holomorphic ball maps in one dimension are rational.
Understanding and classifying these maps is still an active field of research
a century later.
Alexander~\cite{alexander-1977-proper} found that for $n > 1$,
any proper holomorphic self map on the unit ball $\B_n$ in a complex Euclidean space 
$\C^n$ is necessarily an automorphism, hence rational of degree $1$ due to a classic result.
This essentially completed our understanding of such maps between balls 
of the same dimension and inspired many researchers to search for such 
maps between balls of different dimensions for the next fifty years.
It is a deep theorem of Forstnerič~\cite{forstneric-1989-extending} that 
any proper holomorphic map between $\B_n$ and $\B_N$ that extends smoothly enough up to the boundary is rational.
Thus a general goal is to classify rational proper maps between balls 
up to spherical equivalence, that is, up to automorphisms.
An important subclass is to classify the polynomial ones.
D'Angelo~\cite{dangelo-1988-polynomial} has classified all such maps.
The general problem of classifying rational ones is still an open problem.
In this work, we will always assume $n \geq 2$.

Webster~\cite{webster-1979-mapping} was the first to successfully attempt 
the positive codimensional case showing that for $n > 2$ and $N = n + 1$, 
any such map is spherically equivalent to 
the linear embedding map $z \mapsto (z, 0)$, which means that 
such maps are of degree~$1$.
Faran~\cite{faran-1982-maps} complemented this result by showing that
there are four spherical equivalence classes for $n = 2$ and $N = n + 1$
with maximum degree~$3$, 
but it was not clear why this case shows more equivalence classes
than the $n > 2$ case.
Cima and Suffridge~\cite{cima-1983-reflection} conjectured and
later Faran~\cite{faran-1986-linearity} proved that for $N \leq 2n - 2$,
any such map is spherically equivalent 
to the linear embedding map, implying only degree-$1$ maps, 
and indicating a gap in the possible minimum target dimension~$N$.

Huang~\cite{huang-1999-linearity} proved the same result under weaker regularity hypothesis
using the Cartan-Chern-Moser theory~\cite{chern-1974-real},
which led to a series of results 
\cites{huang-2001-mapping,huang-2003-semirigidity,hamada-2005-rational,huang-2006-new,huang-2014-third}
in the same direction in the following decade.
Huang and Ji~\cite{huang-2001-mapping} showed that there can be two equivalence classes for $n \geq 3$ and $N = 2n - 1$ with maximum degree~$2$.
Among other results,
Hamada~\cite{hamada-2005-rational} found all maps for $n \geq 4$ and $N = 2n$
to have a maximum degree of~$2$.
The work of~\cite{huang-2006-new} and~\cite{andrews-2016-mapping}
classified the case $4 \leq n \leq N \leq 3n - 4$ and 
the case $4 \leq n \leq N = 3n - 3$ respectively,
both showing a maximum degree of $2$.
Lebl~\cite{lebl-2011-normal} classified all degree-$2$ such maps into 
uncountably many spherical equivalence classes represented by monomial maps.
More discussion on this subject can be found in
the articles \cites{forstneric-1993-proper,huang-2003-semirigidity,lebl-2024-exhaustion} 
and the book \cite{dangelo-1993-several} 
and references therein.

One way of measuring the complexity of a rational proper map is its degree.
The celebrated result of Forstnerič~\cite{forstneric-1989-extending} also shows that 
any rational proper map between $\B_n$ and $\B_N$ has a degree bounded by a constant
$N^2(N-n+1)$ depending only on the dimensions of the unit balls, but the bound was not sharp.
D'Angelo~\cite{dangelo-2003-sharp} made a conjecture about the degree bound:
that any rational proper map between balls $F \colon \B_n \to \B_N$ has
\[
    \deg F \leq 
    \begin{cases}
        2N - 3          & n = 2 \\
        \frac{N-1}{n-1} & n > 2
    \end{cases},
\]
where both bounds are known to be sharp if true.
The conjecture has been proved for all such monomial maps 
by D'Angelo, Kos and, Riehl~\cite{dangelo-2003-sharp} for $n = 2$ and
by Lebl and Peters~\cites{lebl-2011-polynomials,lebl-2012-polynomials} for any $n \geq 3$.
Meylan~\cite{meylan-2006-degree} showed that $\deg F \leq \frac{N(N-1)}{2}$ for $n = 2$, and D'Angelo and Lebl~\cite{dangelo-2009-complexity} later proved that for any $n \geq 2$, $\deg F \leq \frac{N(N-1)}{2(2n-3)}$.

The proof of the $n \geq 3$ and $N = 2n - 1$ case~\cite{huang-2001-mapping} 
uses the rank of a certain matrix called geometric rank 
$\kappa_0 \in [0, n-1]$ of the map, which is used to measure the degeneracy of the second fundamental form of the map.
See \cref{sec:preliminaries,def:kappa0} for the definition.
Huang~\cite{huang-2003-semirigidity} proved a deep result showing that 
maps with $\kappa_0 < n - 1$ satisfy a semi-linearity property, 
which maps with $\kappa_0 = n - 1$ may not,
making the latter maps more complicated than the former.
He also showed in the same work that 
$N \geq n + \frac{\kappa_0 (2n - \kappa_0 - 1)}{2}$, 
which implies that the minimum target dimension depends on the geometric rank.
This splits the study of these maps into four problems:

\begin{enumerate}[label=(\Alph*)]
    \item 
Study rational proper maps $F \colon \B_n \to \B_N$ with $\kappa_0 < n - 1$ and $N = n + \frac{\kappa_0 (2n - \kappa_0 - 1)}{2}$;
    \item 
Study rational proper maps $F \colon \B_n \to \B_N$ with $\kappa_0 < n - 1$ and $N > n + \frac{\kappa_0 (2n - \kappa_0 - 1)}{2}$;
    \item \label{enum:max-kappa0} 
Study rational proper maps $F \colon \B_n \to \B_N$ with $\kappa_0 = n - 1$ and $N = \frac{n(n+1)}{2}$;
    \item 
Study rational proper maps $F \colon \B_n \to \B_N$ with $\kappa_0 = n - 1$ and $N > \frac{n(n+1)}{2}$.
\end{enumerate}
This, for example, explains why the $n = 2$ and $n > 2$ cases differ for $N = n + 1$.

The work of 
\cites{huang-1999-linearity,huang-2001-mapping,ji-2004-maps,hamada-2005-rational,huang-2006-new,ji-2018-upper} 
deal with $\kappa_0 < n - 1$ in many different settings.
Huang, Ji, and Xu~\cite{huang-2006-new} confirmed D'Angelo conjecture 
for $n \geq 3$ and geometric rank $\kappa_0 = 1$.
The case of maximum geometric rank $\kappa_0 = n - 1$
has mostly been unresolved.
For Problem~\ref{enum:max-kappa0}, 
the $n = 2$ case has been solved by Faran~\cite{faran-1982-maps}, 
where the sharp degree bound turned out to be $3$.
The $n = 3$ case is still unresolved, but Huang, Ji, and Xu~\cite{huang-2005-several} found a degree bound of $4$.
The current work focuses on the degree bound of Problem~\ref{enum:max-kappa0}.
Our main result is as follows.

\begin{theorem}
\label{thm:max-kappa0}
    Let $F \colon \B_n \to \B_N$ be a proper holomorphic map 
    that is $C^3$-smooth up to the boundary with geometric rank $\kappa_0$, 
    $n \geq 2$, and $N = n + \frac{\kappa_0 (2n - \kappa_0 - 1)}{2}$.
    Then $F$ is rational with $\deg F \leq \kappa_0 + 2$.
\end{theorem}

Two immediate corollaries are the following degree bounds for rational proper maps:

\begin{corollary}
    Let $F \colon \B_n \to \B_N$ be a rational proper map with geometric rank $\kappa_0$, 
    $n \geq 2$, and $N = n + \frac{\kappa_0 (2n - \kappa_0 - 1)}{2}$.
    Then $\deg F \leq \kappa_0 + 2$.
\end{corollary}

\begin{corollary}
    Let $n \geq 2$, $N = \frac{n(n+1)}{2}$, and
    $F \colon \B_n \to \B_N$ be a rational proper map 
    with geometric rank $n - 1$.
    Then $\deg F \leq n + 1$.
\end{corollary}

The $\kappa_0 < n - 1$ case has been proved by Ji and Xu~\cite{ji-2004-maps}.

The organization of this paper is as follows.
In \cref{sec:preliminaries}, we set up the preliminaries required 
for our main proof.
In \cref{sec:main-proof}, we prove our main result, except for two claims,
which will be proved in \cref{sec:claim-proof}.

\section{Normal Forms, Geometric Rank, and Degree Bound}
\label{sec:preliminaries}

Our result is based on the normal form and the degree result in \cref{cor:normal-form-2nd-order} which this section will lead to.

\subsection{Associated Maps}
\label{ssec:associated}

We will use a series of associated maps and state a series of normalization for a rational proper map.
Write $\bH_n = \set{(z, w) \in \C^{n-1} \times \C \suchthat \Im w > \norm{z}^2}$
for the Siegel upper half-space and
$\delH_n = \set{(z, w) \in \C^{n-1} \times \C \suchthat \Im w = \norm{z}^2}$
for its boundary, the Heisenberg group.
We parametrize $\delH_n$ by $(z, \bar{z}, u)$ through the map
$(z, \bar{z}, u) \mapsto (z, u + i \norm{z}^2)$, and
for a non-negative integer $m$ and a function $h(z, \bar{z}, u)$ defined
on a small ball $U$ around $0$ in $\delH_n$,
we write $h = \weightedorder(m)$ if 
$\frac{h(tz, t\bar{z}, t^2 u)}{\abs{t}^m} \to 0$ uniformly for $(z,u)$ 
on every compact subset of $U$ as the real number $t \to 0$.

Let $F = (\tilde{f},\tilde{g}) \colon \delH_n \to \delH_N$ be a rational CR map.
Take any $p = (z_0, w_0) \in \delH_n$,
consider $\sigma_p \in \Aut(\bH_n)$ and $\tau_p^F \in \Aut(\bH_N)$ given by
\begin{gather*}
    \sigma_p(z, w) = (z + z_0, w + w_0 + 2i z \cdot \overline{z_0}), \\
    \tau_p^F(z^*, w^*) = (z^* - \tilde{f}(z_0, w_0), w^* - \tilde{g}(z_0, w_0) - 2i z^* \cdot \overline{\tilde{f}(z_0, w_0)}),
\end{gather*}
where $\cdot$ is the standard bilinear product.
Then 
$F_p := \tau_p^F \circ F \circ \sigma_p  \colon \delH_n \to \delH_N$ 
is a rational CR map with
$F_p(0) = 0$.
It follows from Huang's pioneer paper~\cite{huang-1999-linearity}*{Section~4} that 
there are automorphisms
$H_p, G_p \in \Aut(\bH_N)$ and $F_p^* := H_p \circ F_p$ such that 
the rational CR map
$F_p^{**} = (f_p^{**}, \phi_p^{**}, g_p^{**}) := G_p \circ F_p^* \colon \delH_n \to \delH_N$ 
satisfies the following normalization condition:
\begin{gather*}
    f_p^{**} = z + \frac{i}{2} a_p^{\{1\}}(z) w + \weightedorder(3), \\
    \phi_p^{**} = \phi_p^{\{2\}}(z) + \weightedorder(2), \\
    g_p^{**} = w + \weightedorder(4),
\end{gather*}
with 
$\overline{z} \cdot \overline{a_p^{\{1\}}(z)} \norm{z}^2 = \norm{\phi_p^{\{2\}}(z)}^2$,
where we denote by $h^{\{j\}}(z)$ a homogeneous polynomial degree $j$ in $z$, that is,
$h^{\{j\}}(cz) = c^j h^{\{j\}}(z)$ for any $c \in \C$.

\subsection{The Geometric Rank}
\label{ssec:geometric-rank}

The normalization lets us write
$a_p^{\{1\}}(z) = z \mathcal{A}_p$, 
where 
\begin{equation}
\label{eq:geomtric-rank-matrix}
    \mathcal{A}_p = -2i \left( \frac{\partial^2 (f_p)^{**}_\ell}{\partial z_j \partial w} \right)_{1 \leq j, \ell \leq n-1}
\end{equation}
is an $(n-1) \times (n-1)$ Hermitian positive semidefinite matrix.
Write $\rho_n \colon \bH_n \to \B_n$ for the Cayley transformation,
which extends to $\rho_n \colon \partial\bH_n \to \partial\B_n$.

\begin{definition}[Geometric Rank]
\label{def:kappa0}
    We define the geometric rank $\Rk_F(p)$ of 
    a rational CR map $F \colon \delH_n \to \delH_N$ at $p \in \delH_n$ to be the rank of the $\mathcal{A}_p$ from \cref{eq:geomtric-rank-matrix}, 
    and the geometric rank of $F$ to be
    \[
        \kappa_0 = \max_{p \in \delH_n} \Rk_F(p).
    \]
    The \emph{geometric rank} of a rational proper map $F \colon \B_n \to \B_N$ 
    is defined to be the geometric rank of 
    $\rho_N^{-1} \circ F \circ \rho_n$.
\end{definition}

Note that $0 \leq \kappa_0 \leq n - 1$.
Huang~\cite{huang-2003-semirigidity}*{Lemma~2.2~(B)} showed that 
$\kappa_0$ is independent of $H_p$ or $G_p$ and in fact depends only on the spherical equivalence class of $F$.
The geometric rank is used to measure the degeneracy of the second fundamental form of the map.
The minimum geometric rank $\kappa_0 = 0$ is associated with linear fractional maps~\cite{huang-1999-linearity}.

To put the map $F$ into one more normal form, 
we write
$\mathcal{S}_0 = \set{(j, k) \suchthat 1 \leq j \leq \kappa_0, 1 \leq k \leq n - 1, j \leq k}$,
$\mathcal{S}_1 = \set{(j, k) \suchthat
j = \kappa_0 + 1, k \in \set{\kappa_0 + 1, \dots, N - n - P(n, \kappa_0)}}$,
$\mathcal{S} = \mathcal{S}_0 \cup \mathcal{S}_1$,
and
$P(n, \kappa_0) = \frac{\kappa_0 (2n-\kappa_0-1)}{2}$,
number of elements in $\mathcal{S}_0$.
We can use Lemma~3.2 and its proof from Huang~\cite{huang-2003-semirigidity} to show that
if $\Rk_F(0) = \kappa_0$, 
then
$N \geq n + P(n, \kappa_0)$ and
there is an automorphism $\gamma_p \in \Aut(\bH_N)$ such that the rational CR map
$F_p^{***} = (f, \phi, g) := \gamma_p \circ F_p^{**} \colon \delH_n \to \delH_N$ satisfies
the following normalization condition:
\begin{gather*}
    f_j = z_j + \frac{i\mu_j}{2} z_j w + \weightedorder(3), \quad
    \frac{\partial^2 f_j}{\partial w^2}(0) = 0, \quad 
    \mu_j > 0, \quad j = 1, \dots, \kappa_0, \\
    f_j = z_j + \weightedorder(3), \quad
    j = \kappa_0 + 1, \dots, n - 1, \\ 
    g = w + \weightedorder(4), \\
    \phi_{jk} = \mu_{jk} z_j z_k + \weightedorder(2), \\
    \text{where } (j,k) \in \mathcal{S} \text{ with } \mu_{jk} > 0
    \text{ for } (j,k) \in \mathcal{S}_0 \text{ and } \mu_{jk} = 0
    \text{ otherwise}.
\end{gather*}
Moreover, 
$\mu_{jk} = \sqrt{\mu_j + \mu_k}$ for $j, k \leq \kappa_0, j \neq k$,
and $\mu_{jk} = \sqrt{\mu_j}$ for $j \leq \kappa_0 < k \text{ or } j = k \leq \kappa_0$.

\subsection{A Degree Result}

Now let us focus on how to get a degree estimate from the normal forms. 

\begin{definition}[Degree of a Rational Map]
    Let $F = \frac{P}{Q} = \frac{(P_1, \dots, P_N)}{Q}$ be a rational map from $\C^n$ into $\C^N$ in reduced terms.
    We define the \emph{degree} of $F$ to be
    \[
        \deg F = \max(\deg P_1, \dots, \deg P_N, \deg Q).
    \]
\end{definition}

Denote the Segre variety by 
$Q_{(\zeta, \eta)} = \set{(z, w) \suchthat \frac{w - \bar\eta}{2i} = z \cdot \overline{\zeta}}$.
A useful result that we will use to find a degree estimate is
due to Huang and Ji~\cite{huang-2001-mapping}*{Lemma~5.4}:

\begin{proposition}
\label{prop:deg-segre}
Let $F$ be a rational map from $\C^n$ into $\C^N$ in reduced terms
and $K$ a positive integer such that for all $p \in \delH_n$ close to the origin, 
$\deg F|_{Q_p} \leq K$.
Then $\deg F \leq K$.
\end{proposition}

Writing $\beta_p := \gamma_p \circ G_p \circ H_p \circ \tau_p^F$ gives us
$F_p^{***} = \beta_p \circ F \circ \sigma_p$.
Notice that
\[
    \sigma_p(Q_0) = Q_p.
\]
Since all automorphisms of $\Aut(\bH_N)$ are of degree $1$, so is $\beta_p$.
We see that
\[
    \deg F_p^{***}|_{Q_0}
    = \deg \beta_p \circ F \circ \sigma_p|_{Q_0}
    = \deg F \circ \sigma_p|_{Q_0}
    = \deg F|_{Q_p},
\]
which tells us that the condition $\deg F|_{Q_p} \leq K$ from \cref{prop:deg-segre} 
is equivalent to $\deg F_p^{***}|_{Q_0} \leq K$.

We summarize these results in the following.
Write $\phi_{jk}^{(\ell)}$ and $\phi_{jk}^{(0)}$ for the coefficient of $z^\ell w$ and $w^2$ respectively in the Taylor series of $\phi_{jk}$.
We get the following normal form up to 2nd order and degree result that will be useful to find degrees.

\begin{corollary}
\label{cor:normal-form-2nd-order}
Let $F \colon \delH_n \to \delH_N$ be a rational CR map of
geometric rank $\kappa_0$.
Then 
\begin{enumerate}
    \item \label{cor:minimum-N}
$N \geq n + P(n, \kappa_0)$.

    \item \label{cor:normal-form}
For every $p \in \delH_n$, 
$F$ is spherically equivalent to a rational CR map 
$F_p^{***} = (f, \phi, g)$
preserving the origin and
satisfying the following normalization condition:
\begin{gather*}
    f_j = z_j + \lambda_j z_j w + o(2), \\
    \phi_{jk} = \mu_{jk} z_j z_k 
    + \sum_{\ell = 1}^{n-1} \phi_{jk}^{(\ell)} z_\ell w + \phi_{jk}^{(0)} w^2
    + o(2), \\
    g = w + o(2),
\end{gather*}
where 
$\lambda_j = 0 \text{ for } j > \kappa_0, 
\lambda_j \neq 0 \text{ otherwise}$, and
$\mu_{jk} = 0 \text{ for } k \geq j = \kappa_0 + 1,
\mu_{jk} > 0 \text{ otherwise}$.

    \item \label{cor:segre}
Moreover, let $K$ be a positive integer such that 
for all $p \in \delH_n$ close to the origin, 
$\deg F_p^{***}|_{Q_0}  \leq K$.
Then $\deg F \leq K$.
\end{enumerate}
\end{corollary}

\section{Proof of the Main Result}
\label{sec:main-proof}

\subsection{Partial Normalization}

Let $F \colon \B_n \to \B_N$ be a proper holomorphic map 
that is $C^3$-smooth up to the boundary
with geometric rank $\kappa_0$,
$n \geq 2$, and $N = n + \frac{\kappa_0 (2n - \kappa_0 - 1)}{2}$.
Since $0 \leq \kappa_0 \leq n-1$, we get $N \leq \frac{n(n+1)}{2}$ 
and hence 
$F$ is a rational map by~\cite{huang-2005-several}*{Corollary~1.4}.
Let $d$ be the degree of the map $F$.
As the $\kappa_0 < n - 1$ case has been proved in~\cite{ji-2004-maps},
we will assume $\kappa_0 = n - 1$.

The map $F$ is a rational proper map between balls and so extends holomorphically across the boundary $\partial\B_n$ due to 
a well-known result by Cima and Suffridge~\cite{cima-1990-boundary},
and takes $\partial\B_n$ to $\partial\B_N$.
Then the extension $\rho_N^{-1} \circ F \circ \rho_n \colon \delH_n \to \delH_N$,
which we will also call $F$, is a rational CR map of degree~$d$,
as $\rho_n$ and $\rho_N^{-1}$ are rational maps of degree~$1$.
Take any $p \in \delH_n$ near the origin.
Using \cref{cor:normal-form-2nd-order} \ref{cor:normal-form}
with $\kappa_0 = n - 1$,
we get that 
$F$ is spherically equivalent to the origin preserving map 
\[
    F^{***}_p
    = (\tilde{f}, g) 
    = (f, \phi, g) 
    = (f_1, \dots, f_{n-1}, \phi_1, \dots, \phi_{N-n}, g)
\]
with
\begin{gather*}
    f_j = z_j + \lambda_j z_j w + o(2), \\
    \phi_{jk} = \mu_{jk} z_j z_k 
    + \sum_{\ell = 1}^{n-1} \phi_{jk}^{(\ell)} z_\ell w + \phi_{jk}^{(0)} w^2
    + o(2), \\
    g = w + o(2),
\end{gather*}
where 
$\lambda_j \neq 0$ for $j = 1, \dots, n - 1$,
$\mu_{jk} > 0$ for $(j,k) \in \mathcal{S} = \mathcal{S}_0$,
and $\mathcal{S}_1 = \emptyset$.

To prove that $d = \deg F \leq n+1$, it is sufficient to prove that
for all $p \in \delH_n$, $\deg F^{**}_p|_{Q_0} \leq n+1$
because of \cref{cor:normal-form-2nd-order} \ref{cor:segre}.
Here $Q_0 = \set{w = 0}$.

Consider the CR vector fields
\[
    L_k = \diffp{}{{z_k}} + 2i \bar{z}_k \diffp{}{w}
\]
for $k = 1, \dots, n-1$
and complexify these to get
\[
    \cL_k = \diffp{}{{z_k}} + 2i \bar{\zeta}_k \diffp{}{w}.
\]
Write $\epsilon_k = 2i \bar{\zeta}_k$ and compute
\begin{gather*}
    \cL_k = \diffp{}{{z_k}} + \epsilon_k \diffp{}{w}, \\
    \cL_j \cL_k = \diffp{}{{z_j}{z_k}} + 2 \epsilon_j \diffp{}{{z_k}{w}} + 2 \epsilon_k \diffp{}{{z_j}{w}} + \epsilon_j \epsilon_k \diffp{}{{w^2}}
\end{gather*}
for $j, k = 1, \dots, n-1$.
On $\delH_n$, we get the basic equation $\Im g = \norm{\tilde{f}}^2$, that is,
\[
    \Im g(z,w) 
    = \frac{g(z,w) - \overline{g(z,w)}}{2i}
    = \tilde{f}(z,w) \cdot \overline{\tilde{f}(z,w)}.
\]
Complexification gives us
\begin{equation}
\label{eq:formal}
    \frac{g(z,w) - \overline{g(\zeta, \eta)}}{2i}
    = \tilde{f}(z,w) \cdot \overline{\tilde{f}(\zeta, \eta)}
\end{equation}
along any Segre variety $Q_{(\zeta, \eta)}$.

\begin{notation}
For any positive integer $m$,
denote by $[m]$ by the set $\set{1, \dots, m}$.
\end{notation}

\subsection{A Degree Estimate}

We will describe the normalized map $F_p^{***}$ along a Segre variety
and get a degree estimate from there.
Set $n' = n - 1$ and $N' = N - n$.
We apply $\cL_k$ for $k \in [n']$ and $\cL_j \cL_k$ for $(j, k) \in \mathcal{S}_0$ on \cref{eq:formal} to get
\begin{gather*}
    \frac{1}{2i} \cL_k g(z,w) = \cL_k\tilde{f}(z,w) \cdot \overline{\tilde{f}(\zeta,\eta)}, \\
    \frac{1}{2i} \cL_j \cL_k g(z,w) = \cL_j \cL_k\tilde{f}(z,w) \cdot \overline{\tilde{f}(\zeta,\eta)}.
\end{gather*}
Letting $(z,w) = 0$ and $\eta = 0$ gives us
\[
    \frac{1}{2i} (0 - \overline{g(\zeta, 0)})
    = 0 \cdot \overline{\tilde{f}(\zeta,0)},
\]
that is, $\overline{g(\zeta, 0)} = 0$,
and
\begin{gather*}
    \frac{1}{2i}
    \restrict{
    \begin{bmatrix}
        \cL_1 g \\
        \vdots \\
        \cL_{n'} g \\
        \cL_1 \cL_1 g \\
        \vdots \\
        \cL_{n'} \cL_{n'} g
    \end{bmatrix}
    }{(0,0)}
    =
    \restrict{
    \begin{bmatrix}
        \cL_1 \tilde{f} \\
        \vdots \\
        \cL_{n'} \tilde{f} \\
        \cL_1 \cL_1 \tilde{f} \\
        \vdots \\
        \cL_{n'} \cL_{n'} \tilde{f}
    \end{bmatrix}
    }{(0,0)} \transpose{\overline{\tilde{f}(\zeta,0)}}
    =
    \restrict{
    \begin{bmatrix}
        \cL_1 f & \cL_1 \phi \\
        \vdots & \vdots \\
        \cL_{n'} f & \cL_{n'} \phi \\
        \cL_1 \cL_1 f & \cL_1 \cL_1 \phi \\
        \vdots & \vdots \\
        \cL_{n'} \cL_{n'} f & \cL_{n'} \cL_{n'} \phi
    \end{bmatrix}
    }{(0,0)} \transpose{\overline{\tilde{f}(\zeta,0)}}.
\end{gather*}
Define
\[
    \lambda_{jk} =
    \begin{cases}
        2 \mu_{jk}  & j = k \\
        \mu_{jk}    & j \neq k
    \end{cases}
\]
for $(j, k) \in \mathcal{S}_0$.
We will label the components of $\phi$ both by single indices $\ell \in [N-n]$ and double indices $(j, k) \in \mathcal{S}_0$,
and write $\iota (j, k) = \ell$ and $\iota^{-1} (\ell) = (j, k)$.

At $(0,0)$,
we compute
\begin{gather*}
    \cL_k g = \epsilon_k, \\
    \cL_j \cL_k g = 0, \\
    \cL_k f = \delta_k := (0, \dots, 0, 1, 0, \dots, 0) \in \C^{n'}, \\
    \cL_j \cL_k f
    = 0 + \lambda_j \epsilon_k \delta_j + \lambda_k \epsilon_j \delta_k + \epsilon_j \epsilon_k \times 0
    = \lambda_j \epsilon_k \delta_j + \lambda_k \epsilon_j \delta_k \in \C^{n'}, \\
    \cL_k \phi
    = 0 \in \C^{N'}, \\
    \cL_j \cL_k \phi
    = \lambda_{jk} \delta_{\iota(j,k)} 
    + \epsilon_j (\phi_{\iota^{-1}(1)}^{(k)}, \dots, \phi_{\iota^{-1}(N')}^{(k)})
    + \epsilon_k (\phi_{\iota^{-1}(1)}^{(j)}, \dots, \phi_{\iota^{-1}(N')}^{(j)}) \\
    + \epsilon_j \epsilon_k (2\phi_{\iota^{-1}(1)}^{(0)}, \dots, 2\phi_{\iota^{-1}(N')}^{(0)})
\end{gather*}
for $k \in [n']$ and $(j,k) \in \mathcal{S}_0$.

Set $\phi^{(k)} = (\phi_{\iota^{-1}(1)}^{(k)}, \dots, \phi_{\iota^{-1}(N')}^{(k)})$
and $\phi^{(0)} = (\phi_{\iota^{-1}(1)}^{(0)}, \dots, \phi_{\iota^{-1}(N')}^{(0)})$, so that at $(0,0)$,
\[
    \cL_j \cL_k \phi
    = \lambda_{jk} \delta_{\iota(j,k)} 
    + \epsilon_j \phi^{(k)}
    + \epsilon_k \phi^{(j)}
    + 2 \epsilon_j \epsilon_k \phi^{(0)}.
\]
Notice that at $(0,0)$,
\[
    (\cL_k f)_j = I = I_{n'}, \quad
    (\cL_k \phi)_j = 0 = 0_{n' \times N'},
\]
and write
\[
    C = 
    \restrict{
    \begin{bmatrix}
        \cL_1 \tilde{f} \\
        \vdots \\
        \cL_{n'} \tilde{f} \\
        \cL_1 \cL_1 \tilde{f} \\
        \vdots \\
        \cL_{n'} \cL_{n'} \tilde{f}
    \end{bmatrix}
    }{(0,0)}, \quad
    A =
    \restrict{
    \begin{bmatrix}
        \cL_1 \cL_1 f \\
        \vdots \\
        \cL_{n'} \cL_{n'} f
    \end{bmatrix}
    }{(0,0)},
\]
\[
    B =
    \restrict{
    \begin{bmatrix}
        \cL_1 \cL_1 \phi \\
        \vdots \\
        \cL_{n'} \cL_{n'} \phi
    \end{bmatrix}
    }{(0,0)}
    = (\restrict{\cL_j \cL_k \phi}{(0,0)})_{\iota(j,k)}, \text{ and }
    D = 
    \frac{1}{2i}
    \restrict{
    \begin{bmatrix}
        \cL_1 g \\
        \vdots \\
        \cL_{n'} g \\
        \cL_1 \cL_1 g \\
        \vdots \\
        \cL_{n'} \cL_{n'} g
    \end{bmatrix}
    }{(0,0)}
    =
    \frac{1}{2i}
    \begin{bmatrix}
        \epsilon \\
        0
    \end{bmatrix}.
\]
We see that at $(0,0)$, 
$\det B = \prod_{k=1}^{n'} \lambda_{\iota^{-1}(k)} = 2^{n'} \prod_{k=1}^{n'} \mu_{\iota^{-1}(k)} > 0$,
so that near the origin, $\det B > 0$ and $B$ is invertible.
Hence
\[
    C = 
    \begin{bmatrix}
        I & 0 \\
        A & B
    \end{bmatrix}, \quad
    \det C = \det B \neq 0, \quad
    C^{-1} =
    \begin{bmatrix}
        I           & 0 \\
        -B^{-1} A   & B^{-1}
    \end{bmatrix},
\]
and
\[
    \transpose{\overline{\tilde{f}(\zeta,0)}} =
    C^{-1} D
    = 
    \frac{1}{2i}
    \begin{bmatrix}
        \epsilon \\
        -B^{-1} A \epsilon
    \end{bmatrix}
    =
    \frac{1}{2i}
    \frac{
    \begin{bmatrix}
    (\det B) \epsilon \\
    -(\adj B) A \epsilon
\end{bmatrix}
    }{\det B},
\]
which describes $\tilde{f}$ along the Segre variety $Q_0$ in terms of the matrices $A$ and $B$.
We see that $\det B$, $\adj B$, and $A$
are all polynomial maps in $\epsilon$ with
\[
    \deg \det B \leq \prod_{k=1}^{N'} 2 = 2^{N'}, \quad
    \deg \adj B \leq \prod_{k=1}^{N'-1} 2 = 2^{N'-1}, \quad
    \deg A = 1,
\]
and
\[
    \deg (\adj B) A \leq 2^{N'-1} + 1
\]
as polynomial maps in $\epsilon$, so in $\bar{\zeta}$.

We will in fact show in the next section that
\begin{enumerate}
    \item $\deg \det B \leq n$ and
    \item $\deg (\adj B)A \leq n$,
\end{enumerate}
so that 
$\deg 
\begin{bmatrix}
    (\det B) \epsilon \\
    -(\adj B) A \epsilon
\end{bmatrix}
\leq n+1$, giving us $\deg \overline{\tilde{f}(\zeta,0)} \leq n+1$.
Combining this with $\overline{g(\zeta,0)} = 0$, we will have shown that
$\deg \overline{F^{***}_p(\zeta,0)} \leq n+1$, that is, $\deg F^{***}_p|_{Q_0} \leq n+1$, proving our desired result.

\section{A Linear Algebraic Proof of the Claims}
\label{sec:claim-proof}

The proofs of both the claims are completely linear algebraic in nature.

\begin{notation}
\label{notation:homogeneous}
We will decompose a matrix of polynomials $M$ of degree $d$ into the unique homogeneous expansion 
\[
    M = \sum_{j = 0}^d M^{\{j\}},
\]
where $M^{\{j\}}$ is homogeneous of degree $j$, that is,
$M^{\{j\}}(cz) = c^j M^{\{j\}}(z)$ for any $c \in \C$.
\end{notation}

\begin{notation}
\label{notation:matrix}
For any matrix $M$,
we denote by
\begin{itemize}
    \item $M_{jk}$ the $(j,k)$-th element
    \item $M_k$ the column $k$
    \item $M[j]$ the matrix $M$ with row $j$ removed
    \item $M[j, k]$ the matrix $M$ with row $j$ and column $k$ removed.
\end{itemize}
\end{notation}

\subsection{Matrix Structures}

We will look at the columns of linear and quadratic terms of matrix $B$ and those of matrix $A$.
We define the matrix $E \in \C^{N' \times n'}$ and the column vector $e \in \C^{N'}$ by defining $E_{jk}$ and $e_j$ in 
the following way:
Write $(p, q) = \iota^{-1}(j)$, let $\delta$ be the Kronecker delta function, and define
\[
    E_{jk} := \delta_q^k \epsilon_p + \delta_p^k \epsilon_q, \quad
    e_j := 2 \epsilon_p \epsilon_q.
\]
We write $E_k$ for columns of $E$ to get that
\[
    \sum_{k = 1}^{n'} \epsilon_k E_k
    = e.
\]

Using \cref{notation:homogeneous},
write $B = \sum_{k = 0}^2 B^{\{k\}}$.
We see that the quadratic terms of $B$ form the matrix
\begin{align*}
    B^{\{2\}} 
    &= \left(\restrict{\cL_j \cL_k \phi}{(0,0)}\right)_{\iota(j,k)}^{\{2\}} \\
    &= (2 \epsilon_j \epsilon_k \phi^{(0)})_{\iota(j,k)} \\
    &= 2 e \transpose{\phi^{(0)}},
\end{align*}
so that all its columns are multiples of the column matrix $e$,
making $B^{\{2\}}$ a determinant-$0$ rank-$1$ matrix.
Moreover, as 
\[
    e = \sum_{k = 1}^{n'} \epsilon_k E_k,
\]
all columns of $B^{\{2\}}$ are linear combinations of $E_k$s.

The linear terms of $B$ form the matrix
\begin{align*}
    B^{\{1\}}
    &= \left(\restrict{\cL_j \cL_k \phi}{(0,0)}\right)_{\iota(j,k)}^{\{1\}} \\
    &= (\epsilon_j \phi^{(k)} + \epsilon_k \phi^{(j)})_{\iota(j,k)} \\
    &= E_1 \transpose{\phi^{(1)}} + \dots + E_{n'} \transpose{\phi^{(n')}} \\
    &= \sum_{k = 1}^{n'} E_k \transpose{\phi^{(k)}},
\end{align*}
so that all its columns are linear combinations of $E_k$s,
making $B^{\{1\}}$ a determinant-$0$ rank-$n'$ matrix.
As a consequence, columns of $B^{\{k\}}$s are linear combinations of $E_j$s for $k \geq 1$.

Finally
\begin{align*}
    A 
    &= (\restrict{\cL_j \cL_k f}{(0,0)}_{\iota(j,k)} \\
    &= (\lambda_j \epsilon_k \delta_j + \lambda_k \epsilon_j \delta_k)_{\iota(j,k)} \\
    &= 
    \begin{bmatrix}
        \lambda_1 E_1 & \dots & \lambda_{n'} E_{n'}
    \end{bmatrix},
\end{align*}
so that each of its columns $A_k$ is a nonzero multiple of $E_k$.

\subsection{Determinants and Adjugates}

We will need the following two lemmas.

Determinants are multilinear in columns:
Let $M \in \C^{L \times L}$ be a square matrix.
Write $M_\ell$ for columns of $M$ and
decompose a given column $M_j = \sum_{k=1}^K M_j^k$ into finitely many terms.
Then
\begin{align}
\label{eq:additive-det}
\begin{split}
    \det M
    &= \det 
    \begin{bmatrix}
        \displaystyle
        M_1 & \dots & M_{j-1} & \sum_{k=1}^K M_j^k & M_{j+1} & \dots & M_L
    \end{bmatrix} \\
    &= \sum_{k=1}^K
    \det
    \begin{bmatrix}
        M_1 & \dots & M_{j-1} & M_j^k & M_{j+1} & \dots & M_L
    \end{bmatrix}.
\end{split}
\end{align}

Adjugate matrices are not multilinear, but we get the following formula:
\begin{lemma}
\label{lem:additive-adj}
    Let $M \in \C^{L \times L}$ be a square matrix.
    Write $M_\ell$ for columns of $M$ and
    decompose a given column $M_j = \sum_{k=1}^K M_j^k$ into finitely many terms.
    Then
    \begin{align*}
        \adj M
        &= \adj 
        \begin{bmatrix}
            \displaystyle
            M_1 & \dots & M_{j-1} & \sum_{k=1}^K M_j^k & M_{j+1} & \dots & M_L
        \end{bmatrix} \\
        &= \sum_{k=1}^K
        \adj
        \begin{bmatrix}
            M_1 & \dots & M_{j-1} & M_j^k & M_{j+1} & \dots & M_L
        \end{bmatrix} \\
        &-
        (K - 1) \adj
        \begin{bmatrix}
            M_1 & \dots & M_{j-1} & 0 & M_{j+1} & \dots & M_L
        \end{bmatrix}.
    \end{align*}
\end{lemma}

\begin{proof}
The key to the proof is that all the three matrices under the $\adj$ operator
on the middle and right sides of the formula stay the same if we remove their $j$-th columns.
Using \cref{notation:matrix}, we get that
the $(\ell, m)$-th element of $\adj M$ is
\begin{align*}
    &(-1)^{\ell+m} \det M[m, \ell] \\
    &= (-1)^{\ell+m} \det
    \begin{bmatrix}
        \displaystyle
        M_1 & \dots & M_{j-1} & \sum_{k=1}^K M_j^k & M_{j+1} & \dots & M_L
    \end{bmatrix}
    [m, \ell] \\
    &= 
    \begin{cases}
        (-1)^{j+m}
        \det
        \begin{bmatrix}
            \displaystyle
            M_1 & \dots & M_{j-1} & \sum_{k=1}^K M_j^k & M_{j+1} & \dots & M_L
        \end{bmatrix}
        [m,j],
        & \text{if } \ell = j \\
        (-1)^{\ell+m}
        \det 
        \begin{bmatrix}
            \displaystyle
            M_1 & \dots & M_{j-1} & \sum_{k=1}^K M_j^k & M_{j+1} & \dots & M_L
        \end{bmatrix}
        [m,\ell],
        & \text{if } \ell \neq j
    \end{cases} \\
    &=
    \begin{cases}
        (-1)^{j+m}
        \det
        \begin{bmatrix}
            \displaystyle
            M_1 & \dots & M_{j-1} & 0 & M_{j+1} & \dots & M_L
        \end{bmatrix}
        [m,j],
        & \text{if } \ell = j \\
        (-1)^{\ell+m}
        \sum_{k=1}^K
        \det 
        \begin{bmatrix}
            \displaystyle
            M_1 & \dots & M_{j-1} & M_j^k & M_{j+1} & \dots & M_L
        \end{bmatrix}
        [m,\ell],
        & \text{if } \ell \neq j
    \end{cases},
\end{align*}
where the last case of the last expression uses \cref{eq:additive-det}.
On the other hand, 
the $(\ell, m)$-th element of 
the sum of the adjugate matrices in the formula is
\begin{align*}
    &\sum_{k=1}^K 
    (-1)^{\ell+m} \det
    \begin{bmatrix}
        M_1 & \dots & M_{j-1} & M_j^k & M_{j+1} & \dots & M_L
    \end{bmatrix}    
    [m, \ell] \\
    &= 
    \sum_{k=1}^K
    \begin{cases}
        (-1)^{j+m}
        \det
        \begin{bmatrix}
            \displaystyle
            M_1 & \dots & M_{j-1} & M_j^k & M_{j+1} & \dots & M_L
        \end{bmatrix}
        [m,j],
        & \text{if } \ell = j \\
        (-1)^{\ell+m}
        \det 
        \begin{bmatrix}
            \displaystyle
            M_1 & \dots & M_{j-1} & M_j^k & M_{j+1} & \dots & M_L
        \end{bmatrix}
        [m,\ell],
        & \text{if } \ell \neq j
    \end{cases} \\
    &=
    \begin{cases}
        \sum_{k=1}^K
        (-1)^{j+m}
        \det
        \begin{bmatrix}
            \displaystyle
            M_1 & \dots & M_{j-1} & 0 & M_{j+1} & \dots & M_L
        \end{bmatrix}
        [m,j],
        & \text{if } \ell = j \\
        \sum_{k=1}^K
        (-1)^{\ell+m}
        \det 
        \begin{bmatrix}
            \displaystyle
            M_1 & \dots & M_{j-1} & M_j^k & M_{j+1} & \dots & M_L
        \end{bmatrix}
        [m,\ell],
        & \text{if } \ell \neq j
    \end{cases}.
\end{align*}
This means that the $(\ell,m)$-th element of 
\[
    \sum_{k=1}^K
    \adj
    \begin{bmatrix}
        M_1 & \dots & M_{j-1} & M_j^k & M_{j+1} & \dots & M_L
    \end{bmatrix}
    - \adj M
\]
equals
\begin{align*}
    &\begin{cases}
        (K-1)
        (-1)^{j+m}
        \det
        \begin{bmatrix}
            \displaystyle
            M_1 & \dots & M_{j-1} & 0 & M_{j+1} & \dots & M_L
        \end{bmatrix}
        [m,j],
        & \text{if } \ell = j \\
        0,
        & \text{if } \ell \neq j
    \end{cases} \\
    &= 
    (K - 1)
    \begin{cases}
        (-1)^{j+m}
        \det
        \begin{bmatrix}
            \displaystyle
            M_1 & \dots & M_{j-1} & 0 & M_{j+1} & \dots & M_L
        \end{bmatrix}
        [m,j],
        & \text{if } \ell = j \\
        (-1)^{\ell+m}
        \det 
        \begin{bmatrix}
            \displaystyle
            M_1 & \dots & M_{j-1} & 0 & M_{j+1} & \dots & M_L
        \end{bmatrix}
        [m,\ell],
        & \text{if } \ell \neq j
    \end{cases},
\end{align*}
which is the $(\ell,m)$-the element of 
$
(K - 1)
\adj
\begin{bmatrix}
    \displaystyle
    M_1 & \dots & M_{j-1} & 0 & M_{j+1} & \dots & M_L
\end{bmatrix}
$.
\end{proof}

\begin{lemma}
\label{lem:adjugate}
    Suppose that $M \in \C^{L \times L}$ is a nonzero square matrix, and
    there are an integer $K$ with $0 < K < L$,
    linearly independent column vectors $U_1, \dots, U_K \in \C^L$ forming the matrix $U$, linearly independent column vectors $V_1, \dots, V_K \in \C^K$ forming the matrix $V$, and 
    column vectors $T_1, \dots, T_{L-K} \in \C^L$ forming the matrix $T$
    such that 
    $M = 
    \begin{bmatrix}
        \sum_{k = 1}^K U_k \transpose{V_k} & T_1 & \dots & T_{L-K}
    \end{bmatrix} = 
    \begin{bmatrix}
        U \transpose{V} & T
    \end{bmatrix}
    $.
    Then 
    \begin{enumerate}
        \item \label{lem:adjugate-rows}
    For all $K < j \leq L$, row $j$ of $(\adj M) U$ is the zero vector, or equivalently, 
    for all $1 \leq k \leq K < j \leq L$, the $j$-th element of $(\adj M) U_k$ is zero.
        \item \label{lem:nullity-1}
    If also rank of $M$ is $L - 1$,
    then $(\adj M) U = 0$, or equivalently, 
    $(\adj M) U_k = 0$ for all $k \in [K]$.
    \end{enumerate}
\end{lemma}

\begin{proof}
    As the columns of the square matrix $V$ are linearly independent, the reduced row echelon form of $V$ is the identity matrix $I = I_K$,
    that is, there is an invertible matrix $W \in \C^{K \times K}$ with
    $\transpose{W} V = I$.
    This gives us 
    \[
        U \transpose{V} W 
        = U I
        = U.
    \]

    \begin{enumerate}
        \item 
    Write $M[\ell, j]$ to mean the matrix $M$ with row $\ell$ and column $j$ removed using \cref{notation:matrix}.
    Now for all $1 \leq k \leq K < j \leq L$,
    \begin{align*}
        \text{row $j$ of } (\adj M) M_k
        &= (\text{row $j$ of } \adj M) M_k \\
        &= \sum_{\ell = 1}^L (-1)^{j + \ell}
        \det
        \begin{bmatrix}
            M[1, j] & \dots & M[L, j]
        \end{bmatrix}
        M_k \\
        &= 
        \det
        \begin{bmatrix}
            M_1 & \dots & M_k & \dots & M_{j-1} & M_{k} & M_{j+1} & \dots & M_L
        \end{bmatrix} \\
        &= 0.
    \end{align*}
    
    Since 
    \begin{align*}
        \text{row $j$ of } (\adj M) U \transpose{V} 
        &= \text{row $j$ of } (\adj M) 
            \begin{bmatrix}
                M_1 & \dots & M_K
            \end{bmatrix} \\
        &= \text{row $j$ of }
            \begin{bmatrix}
                (\adj M) M_1 & \dots & (\adj M) M_K
            \end{bmatrix} \\
        &= 0,
    \end{align*}
    we get 
    \[
        0 = \text{row $j$ of } (\adj M) 
        U \transpose{V} W        
        = \text{row $j$ of } (\adj M) U.
    \]
    Hence row $j$ of $(\adj M) U = 0$.

        \item
    As $\rank M = L - 1$, we get that $\rank \adj M = 1$,
    and there are $x \in \ker M$ and $y \in \ker \transpose{M}$ such that $\adj M = x \transpose{y}$.

    Since $\transpose{y} M = 0$, we get
    \begin{align*}
        0
        &= x \transpose{y} M W \\
        &= (\adj M)
        \begin{bmatrix}
            U \transpose{V} W & T W
        \end{bmatrix} \\
        &= 
        \begin{bmatrix}
            (\adj M) U & (\adj M) T W
        \end{bmatrix}.
    \end{align*}

    Hence $(\adj M) U = 0$. 
    \qedhere
    \end{enumerate}
\end{proof}

\subsection{Claims}

Now we are ready to prove both our claims.

\begin{proposition}
    If $B$ and $n$ are as before, then $\deg \det B \leq n$.
\end{proposition}

\begin{proof}
Repetitive use of \cref{eq:additive-det} on all columns of $B$ gives us
\[
    \det B 
    = \sum_{0 \leq \sum_{k = 0}^{N'} i_k \leq 2N'}
    \det
    \begin{bmatrix}
        B_1^{\{i_1\}} & \dots & B_{N'}^{\{i_{N'}\}}
    \end{bmatrix}
    = \sum_I \det B_I,
\]
where 
$I := (i_1, \dots, i_{N'})$ and
$
B_I := 
\begin{bmatrix}
    B_1^{\{i_1\}} & \dots & B_{N'}^{\{i_{N'}\}}
\end{bmatrix}
$.
This tells us that
\[
    \deg \det B_I \leq \sum_k i_k
\]
for all $I$ and
\[
    \deg \det B \leq \max_I \deg \det B_I.
\]

Notice that each column of each matrix $B_I$ is homogeneous.
The idea is that if there are enough constant columns, the degree of the determinant is low enough,
and if there are too many nonconstant columns, these must be linearly dependent.

Take any $I$, 
put $M = B_I$,
write $n_{\geq 1} = \#\set{k \suchthat i_k \geq 1}$ for number of columns of $M$ with degree at least $1$, and
$n_j = \#\set{k \suchthat i_k = j}$ for number of columns of $M$ with degree $j$.
This gives us 
$\sum_k i_k = 0 \cdot n_0 + 1 \cdot n_1 + 2 \cdot n_2 = n_{\geq 1} + n_2$.

As columns of $B^{\{k\}}$s are linear combinations of $n' = n - 1$ columns $E_j$s for $k \geq 1$, so are $M^{\{k\}}$s.

\subsection*{Case 1}
First assume that $n_{\geq 1} \geq n$.
Then there are at least $n$ columns $M_k$ which are 
linear combinations of $n - 1$ columns $E_j$s, so that 
$\det M = 0$.

\subsection*{Case 2}
Now assume that $n_2 \geq 2$.
Then there are at least two columns $M_k$ which are multiples of $e$, so that $\det M = 0$.

\subsection*{Case 3}
Finally assume that $n_{\geq 1} \leq n - 1$ and $n_2 \leq 1$.
Then
\[
    \deg \det M \leq n_{\geq 1} + n_2 \leq (n - 1) + 1 = n.
\]

Hence $\deg \det B_I \leq n$ for all $I$ and so $\deg \det B \leq n$.
\end{proof}

\begin{proposition}
\label{prop:adj-B-A}
    If $B$, $A$, and $n$ are as before, then $\deg (\adj B) A \leq n$.
\end{proposition}

\begin{proof}
Repetitive use of \cref{lem:additive-adj} on all columns of $B$ gives us
\[
    \adj B 
    = \sum_{i_1, \dots, i_{N'}}
    \text{constant } \cdot \adj
    \begin{bmatrix}
        B_1^{\{i_1\}} & \dots & B_{N'}^{\{i_{N'}\}}
    \end{bmatrix}
    = \sum_I \text{constant } \cdot \adj B_I,
\]
where 
$I := (i_1, \dots, i_{N'})$,
$B_I := 
\begin{bmatrix}
    B_1^{\{i_1\}} & \dots & B_{N'}^{\{i_{N'}\}}
\end{bmatrix},
$
and we also allow $i_k$ to be $-\infty$ to mean that $B_k^{\{-\infty\}}$ is the zero column vector from \cref{lem:additive-adj}.
This tells us that
\[
    \deg \adj B_I 
    \leq \max_{i_j \geq 0} \Big( \sum_k i_k - i_j \Big)
    = \sum_k i_k - \min_{i_j \geq 0} i_j
\]
for all $I$ and
\[
    \deg \adj B \leq \max_I \adj \det B_I.
\]

Once again,
each column of each matrix $B_I$ is homogeneous.
The idea is that if there are enough constant columns, the degree of the determinant is low enough,
and if there are too many nonconstant columns, these must be linearly dependent.
The proof is more technical than the previous one.

Take any $I$ and
put $M = B_I$.
Using \cref{notation:matrix},
the elements of $\adj M$ are given by
\[
    (\adj M)_{kj} = (-1)^{k+j} \det M[j,k], \quad
    M[j,k] = 
    \begin{bmatrix}
        M_1[j] & \dots & \widehat{M_k[j]} & \dots & M_{N'}[j]
    \end{bmatrix}.
\]
Then
\[
    \deg \det M[j,k] \leq \sum_{\ell \geq 0} i_\ell - i_k
\]
for all $j, k$,
\[
    \deg \adj B_I 
    \leq \max_{j,k} \deg \det M[j,k]
    = \max_k \deg \text{ row $k$ of } \adj M,
\]
and
\[
    \deg (\adj B_I) A
    \leq \max_k \deg (\text{row $k$ of } (\adj M) A)
\]
for all $I$.

Write $n_{\geq 1} = \#\set{k \suchthat i_k \geq 1}$ for number of columns of $M$ with degree at least $1$, and for $j \geq 0$,
$n_j = \#\set{k \suchthat i_k = j}$ for number of columns of $M$ with degree $j$.
This gives us 
$\sum_{k \geq 0} i_k = 0 \cdot n_0 + 1 \cdot n_1 + 2 \cdot n_2 = n_{\geq 1} + n_2$.

As columns of $B^{\{k\}}$s are linear combinations of $n' = n - 1$ columns $E_j$s for $k \geq 1$, so are $M^{\{k\}}$s.
So columns of $M^{\{k\}}[\ell]$s are linear combinations of $n - 1$ columns $E_j[\ell]$s for $k \geq 1$.

\subsection*{Case 1}
First assume that $n_{\geq 1} \geq n$.
Then there are at least $n$ columns $M_k$ which are 
linear combinations of $n - 1$ columns $E_j$s, so that 
nullity of $M$ is at least $n - (n-1) = 1$.
Therefore, $\rank M \leq N' - 1$.
If $\rank M = N' - 1$,
we get $(\adj M) E_k = 0$ for all $k \in [n']$ by \cref{lem:adjugate}~\ref{lem:nullity-1},
so that $(\adj M) A = 0$.
If $\rank M \leq N' - 2$, we get $\adj M = 0$.

\subsection*{Case 2}
Finally assume that $n_{\geq 1} \leq n - 1$.
Fix $k \in [N']$.
For $i_k \leq 0$,
row $k$ of $(\adj M) E_j = 0$ for all $j \in [n']$ by \cref{lem:adjugate}~\ref{lem:adjugate-rows},
so that row $k$ of $(\adj M) A = 0$.
Now let $i_k \geq 1$.
We consider two subcases.

If $n_2 - i_k \leq 0$, then
\[
    \deg \det M[j,k]
    \leq \sum_{\ell \geq 0} i_\ell - i_k 
    = n_{\geq 1} + n_2 - i_k 
    \leq (n - 1) - 0 
    = n - 1
\]
for all $j$, so that
\[
    \deg (\text{row $k$ of } (\adj M) A) 
    \leq (n-1) + 1 
    = n.
\]
Otherwise $n_2 - i_k \geq 1$, so that for all $j$,
there are at least two columns of $M[j,k]$ which are multiples of $e[k]$, 
so that $\det M[j,k] = 0$ and row $k$ of $(\adj M) A = 0$.

Hence $\deg (\adj B_I) A \leq n$ for all $I$ and so $\deg (\adj B) A \leq n$.
\end{proof}

This completes the proof of the main result \cref{thm:max-kappa0}.

\bibliographystyle{amsplain}
\bibliography{main}

\end{document}